\documentclass[12pt]{amsart}
\usepackage{amsthm,amscd,amsmath,amssymb,amsfonts,verbatim,color,todonotes}
\usepackage[cmtip, all]{xy}

\usepackage{graphicx}
\usepackage[left=1.2in, right=1.2in, top=1.2in, bottom=1in, includefoot, headheight=13.6pt]{geometry}

\newtheorem{thm}{Theorem}[section]
\newtheorem{prop}[thm]{Proposition}
\newtheorem{lem}[thm]{Lemma}

\newtheorem{conj}[thm]{Conjecture}
\theoremstyle{definition}

\newtheorem{defn}[thm]{Definition}
\theoremstyle{remark}
\newtheorem{remk}[thm]{Remark}
\newtheorem{remks}[thm]{Remarks}

\newtheorem{exm}[thm]{Example}
\newtheorem{exms}[thm]{Examples}
\newtheorem{notat}[thm]{Notation}
\numberwithin{equation}{section}

{\hfill$\square$\end{defn}}
{\hfill$\square$\end{remk}}
{\hfill$\square$\end{remks}}
{\hfill$\square$\end{exm}}
{\hfill$\square$\end{exms}}
{\hfill$\square$\end{notat}}

\newcommand{\lemref}{Lemma~\ref}

\newcommand{\sT}{{\mathcal T}}

\newcommand{\A}{{\mathbb A}}

\renewcommand{\P}{{\mathbb P}}
\newcommand{\Q}{{\mathbb Q}}

\newcommand{\Z}{{\mathbb Z}}

\newcommand{\fm}{{\mathfrak m}}
\newcommand{\fp}{{\mathfrak p}}

\newcommand{\inj}{\hookrightarrow}
\newcommand{\red}{{\rm red}}

\newcommand{\Spec}{{\rm Spec \,}}

\newcommand{\ds}{{/\kern-3pt/}}

\renewcommand{\red}{{\rm red}}

\renewcommand{\dim}{\text{\rm dim}}

\newcommand{\tuborg}{\left\{\begin{array}{ll}}
\newcommand{\sluttuborg}{\end{array}\right.}

\newcommand{\al}{\alpha}
\newcommand{\bb}{\mathbb}

\newcommand{\blob}{\bullet}
\newcommand{\bor}{\partial}

\newcommand{\into}{\hookrightarrow}
\newcommand{\isoto}{\stackrel{\simeq}{\to}}

\newcommand{\quis}{\stackrel{\sim}{\to}}
\newcommand{\roi}{\mathcal{O}}

\newcommand{\sub}[1]{{\mbox{\scriptsize #1}}}

\newcommand{\To}{\longrightarrow}

\newcommand{\xto}{\xrightarrow}

\renewcommand{\frak}{\mathfrak}
\newcommand{\indlim}{\varinjlim}
\renewcommand{\tilde}{\widetilde}
\renewcommand{\Im}{\operatorname{Im}}
\renewcommand{\ker}{\operatorname{Ker}}

\begin{document}
\title{Analogues of Gersten's conjecture for singular schemes}
\author{Amalendu Krishna, Matthew Morrow}

\address{School of Mathematics, Tata Institute of Fundamental Research,  
1 Homi Bhabha Road, Colaba, Mumbai, India}
\email{amal@math.tifr.res.in}
\address{Mathematisches Institut, Universit{\"a}t-Bonn, Endenicher Allee 60,
53115 Bonn, Germany}
\email{morrow@math.uni-bonn.de}



\maketitle

\begin{center}
To appear in {\em Selecta Mathematica}.
\end{center}

\begin{abstract}
We formulate analogues, for Noetherian local $\bb Q$-algebras which are not necessarily regular, of the injectivity part of Gersten's conjecture in algebraic $K$-theory, and prove them in various cases. 
Our results suggest that the
algebraic $K$-theory of such a ring should be detected by combining the algebraic
$K$-theory of both its regular locus and the infinitesimal 
thickenings of its singular locus.


\end{abstract}

\section{Introduction}\label{sec:Intro}
If $A$ is a regular local ring, then Gersten's conjecture, 
which is a theorem if $A$ contains a field, predicts that the map 
$K_n(A)\to K_n({\rm Frac} A)$ is injective for all $n\ge0$, where 
${\rm Frac} A$ denotes 
the field of fractions of $A$. The aim of this note is to explore certain 
analogues of this injectivity conjecture in the case when $A$ is singular, by 
taking into account nilpotent thickenings of its singular locus. In the case 
that $A$ is one-dimensional, this problem was first considered by the first 
author when $n=2$ \cite{Krishna2005}, and later by the second author in 
general \cite{Morrow_Birelative_dim1}. In this note, we extend some of our 
earlier results to higher dimensions. We hope that our results indicate the existence of, and stimulate work towards, a more general, undiscovered framework for a form of Gersten's conjecture in the presence of singularities.

An illustrative example of such an analogous injection is our
following result for cone singularities.

\begin{thm}\label{thm:Cone-Intro}
Let $k$ be a field of characteristic zero and $Y \inj \P_k^N$ a smooth 
projective variety; let $C \inj \A_k^{N+1}$ be the cone over $Y$ and 
$(A,\frak m)$ the local ring at the unique singular point of $C$.
Then, for any $n\ge0$, the map 
\[K_n(A)\longrightarrow K_n(A/\frak m^r)\oplus K_n(\Spec(A)\setminus\{\frak m\})\] 
is injective for $r\gg 0$.
\end{thm}

Unfortunately, as we shall see in Example \ref{exm:failure}, 
the conclusion of the theorem does not hold for general isolated singularities, 
even in dimension one. To obtain a conjecture which is plausible in general, 
one should replace $\Spec(A) \setminus\{\frak m\}$ by a resolution of 
singularities $X\to \Spec(A)$, and then consider $K$-groups relative to the 
exceptional locus and fibre.

\begin{conj}\label{conj:General-Intro}
Let $A$ be a Noetherian local $\Q$-algebra, $I\subseteq A$ an ideal, 
$n\in\Z$, and
\[
\xymatrix@C1pc{
Y \ar[r] \ar[d] & X \ar[d] \\
\Spec(A/I) \ar[r] & \Spec(A)}
\]
an abstract blow-up square in which $X$ is regular. 
Then the canonical map 
\[
K_n(A,I) \To K_n(A/I^r,I/I^r) \oplus 
K_n(X,Y_{\red})
\] 
is injective for $r\gg0$.
\end{conj}

The conjecture is true by elementary $K$-theory if $n=0$ or $1$; 
see Example \ref{example01}. 
We prove the following cases of this conjecture in this paper.

\begin{thm}\label{thm:Main-1}
The conjecture is true if $A$ is quasi-excellent and
\begin{enumerate}
\item $n = 2$; or
\item $Y_{\red}$ is regular and $n \ge 2$.
\end{enumerate}
\end{thm}

In both cases of the theorem, we prove the stronger assertion
that the canonical sequence of pro abelian groups 
\begin{equation*}
0\To K_n(A,I)\To \{K_n(A/I^r,I/I^r)\}_r\oplus K_n(X,Y_{\red})
\To \{K_n(r Y,Y_{\red})\}_r\To 0
\end{equation*}
is short exact, where $rY:=X\times_AA/I^r$. Indeed, pro cdh descent \cite[Theorem~3.7]{Morrow_pro_cdh_descent} for 
$K$-theory implies 
the existence of a long exact Mayer--Vietoris sequence of pro abelian groups 
\begin{equation*}
\cdots \stackrel\bor\To K_n(A,I)\To \{K_n(A/I^r,I/I^r)\}_r\oplus K_n(X,Y_{\red})
\To \{K_n(r Y,Y_{\red})\}_r\To\cdots
\end{equation*}
and therefore the injectivity assertion of the conjecture is equivalent to the 
vanishing of the boundary maps $\bor$. In the cases of the theorem we can show 
firstly, using pro Hochschild--Kostant--Rosenberg theorems, that 
$\{K_{n+1}(rY,Y_\sub{red})\}_r$ is supported in Adams degrees $>n$ and secondly, 
by classical results of Soul\'e \cite{Soule1985}
and Nesterenko--Suslin \cite{Suslin1989}, that $K_n(A,I)$ is 
supported in Adams degrees $\le n$ up to bounded torsion; see Section \ref{section_preliminary} for 
the details. This forces $\bor$ to be zero.

\subsection{Notations and hypotheses}\label{sec:Note}
We work primarily in the generality of quasi-excellent, Noetherian 
$\bb Q$-algebras, since restricting attention to finite type algebras over a 
characteristic zero field would only slightly simplify some of the proofs. All rings appearing are commutative and Noetherian.

All (relative) Hochschild, (relative) cyclic and Andr{\'e}-Quillen homology
groups will be assumed to be over $\Q$ unless we specify the base ring
explicitly. 

The Adams eigenspaces of a $K$-group $K_n(A)$ are denoted by $K_n^{(i)}(A):=\{x\in K_n(A):\psi^k(x)=k^ix\text{ for all }k\ge0\}$, and similarly for relative groups, for schemes, and for Hochschild and cyclic homology.

Pro abelian groups (always indexed over $\bb N$) are denoted by $\{G_r\}_r$. We will repeatedly use, without explicit mention, that if $G$ is an abelian group and $G\to \{G_r\}_r$ is a map of pro abelian groups (i.e., there are compatible maps $G\to G_r$ for all $r\ge 1$), then $G\to \{G_r\}_r$ is injective if and only if $G\to G_r$ is injective for $r\gg0$.

\subsection*{Acknowledgements}
The second author would like to thank the Tata Institute of Fundamental 
Research for its hospitality during a visit in November 2012.
The authors would like to thank the anonymous referee for carefully reading 
the paper and suggesting many improvements.

\section{Injectivity assertions for $K$-groups of desingularizations}\label{section_assertions}
In this section, we state various injectivity assertions
for the algebraic $K$-theory of singular local schemes in terms of their
desingularizations. We explain the relations between these assertions
and give examples showing that some of them do not hold in general.

Let $A$ be a Noetherian local ring and let $I \subset A$ be an ideal
such that $V(I)$ contains the singular locus of $\Spec(A)$.
We shall say that a commutative diagram
\begin{equation}\label{eqn:ABlow-up}
\xymatrix@C1pc{
Y \ar[r] \ar[d] & X \ar[d]^{f} \\
\Spec(A/I) \ar[r] & \Spec(A)}
\end{equation}
of schemes is {\em a resolution square}, or {\em resolution of $A$}, if and only if it is an abstract blow-up square (i.e., it is Cartesian, $f$ is proper, and the map $f: X \setminus Y \to \Spec(A) \setminus V(I)$ is an isomorphism) and $X$ is regular.

Fixing $n\in\bb Z$, we may then consider the following statements, each of which 
asserts in some sense that the $n^\sub{th}$ $K$-group $K_n$ is captured by a 
combination of generic regular and nilpotent singular data:

\begin{enumerate}
\item[(1)$_n$] the map $K_n(A)\to K_n(A/I^r)\oplus K_n(X)$ is injective for 
$r\gg0$;
\item[(2)$_n$] the map $K_n(A,I)\to K_n(A/I^r,I/I^r)\oplus K_n(X,Y_{\red})$ is 
injective for $r\gg0$;
\item[(3)$_n$] the map $K_n(A)\to K_n(A/I^r)\oplus K_n(\Spec A\setminus V(I))$ 
is injective for $r\gg0$;
\item[(4)$_n$] the map $K_n(X)\to K_n(rY)\oplus K_n(X\setminus Y)$ is 
injective for $r\gg0$, where $rY:=X\times_{A} A/I^r$.
\end{enumerate}
Although assertion (3)$_n$ does not require the existence of the resolution 
square, it appears to be difficult to study in any other way.

Associated to the resolution square there is a commutative diagram of 
pro spectra 
\begin{equation}\label{eqn:ABlow-up-K}
\xymatrix@C1pc{
K(A)\ar[r]\ar[d] & K(X)\ar[d]\\
\{K(A/I^r)\}_r\ar[r] & \{K(rY)\}_r.}
\end{equation}
We will say that the square (\ref{eqn:ABlow-up}) satisfies the 
{\em pro Mayer--Vietoris property} in $K$-theory if and only if the square (\ref{eqn:ABlow-up-K}) of 
pro spectra is homotopy Cartesian (concretely, this means that the associated 
pro abelian relative $K$-groups are isomorphic). This is known to be true if 
$A$ is a quasi-excellent $\bb Q$-algebra 
\cite[Thm.~3.7]{Morrow_pro_cdh_descent}, or 
if $A$ is essentially of finite type over an infinite perfect field having 
strong resolution of singularities \cite[Thm.~3.7]{Morrow_pro_cdh_descent}, or 
if $X\to \Spec A$ is a finite morphism (a consequence of 
\cite[Corol.~0.4]{Morrow_pro_H_unitality}), and 
conjecturally it is true in general for any abstract blow-up square of Noetherian schemes. 
Assuming that the square (\ref{eqn:ABlow-up}) satisfies the 
pro Mayer--Vietoris property in $K$-theory, there are resulting long exact 
Mayer--Vietoris sequences of pro abelian groups:
\begin{equation}\label{eqn:pro-ex-1}
\cdots\To K_n(A)\To \{K_n(A/I^r)\}_r\oplus K_n(X)\To \{K_n(r Y)\}_r\To
\cdots
\end{equation}
and
\begin{equation}\label{eqn:pro-ex-2}
\cdots\To K_n(A,I)\To \{K_n(A/I^r,I/I^r)\}_r\oplus K_n(X,Y_\sub{red})\To 
\{K_n(r Y,Y_\sub{red})\}_r\To\cdots.
\end{equation}

\begin{lem}\label{lemma_equivalences}
For a resolution of $A$ as in (\ref{eqn:ABlow-up}), and $n\in\bb Z$, the following 
implications hold:
\begin{enumerate}
\item[(i)] If the map $K_{n+1}(A)\to K_{n+1}(A/I)$ is surjective, 
then (1)$_n$ $\implies$ (2)$_n$.
\item[(ii)] (1)$_n$\&(4)$_n$ $\implies$ (3)$_n$ $\implies$ (1)$_n$.
\end{enumerate}
If we assume moreover that the resolution square satisfies the pro Mayer--Vietoris 
property, then the following implications also hold:
\begin{enumerate}
\item[(iii)] If the map $K_{n+1}(X)\to K_{n+1}(Y_{\rm red})$ is surjective, then 
(2)$_n$ $\implies$ (1)$_n$.
\item[(iv)] (3)$_n$ $\implies$ (4)$_n$.
\end{enumerate}
\end{lem}
\begin{proof}
The claims (i) and (ii) are completely elementary, noting in (i) that the assumption implies $K_n(A,I)\subseteq K_n(A)$ and in (ii) that $X\setminus Y=\Spec A\setminus V(I)$. Now assume that the resolution square has the pro Mayer--Vietoris property.

(iii):  By~\eqref{eqn:pro-ex-2} and (2)$_n$, the map 
\[
\{K_{n+1}(A/I^r,I/I^r)\}_r\oplus K_{n+1}(X,Y_\sub{red})\to 
\{K_{n+1}(rY,Y_\sub{red})\}_r
\]
is surjective. But the assumed surjectivity of 
$K_{n+1}(X)\to K_{n+1}(Y_\sub{red})$ implies that the maps 
$K_{n+1}(X)\to K_{n+1}(rY)$ and $K_{n+1}(rY,Y_\sub{red})\to K_{n+1}(rY)$ are 
jointly surjective for any $r\ge1$, and so it follows that 
$\{K_{n+1}(A/I^r)\}_r\oplus K_{n+1}(X)\to \{K_{n+1}(rY)\}_r$ is also surjective. 
Now~\eqref{eqn:pro-ex-1} completes the proof of (1)$_n$.

(iv): It follows from~\eqref{eqn:pro-ex-1} that the map 
\[
\{\ker(K_n(A)\to K_n(A/I^r))\}_r\To\{\ker(K_n(X)\to K_n(rY))\}_r
\] 
is surjective, after which the implication is an elementary consequence of 
the identification $X\setminus Y=\Spec A\setminus V(I)$.
\end{proof}

\begin{exm}[$n=0,1$]\label{example01}
We claim that (1)$_1$, (2)$_1$, and (3)$_1$ are true. Indeed, the kernel 
$\Lambda$ of the restriction map 
\[
A=\Gamma(\Spec A,\roi_{\Spec A})\To \Gamma(\Spec A\setminus V(I),\roi_{\Spec A})
\] is supported on $V(I)$, hence annihilated by a power of $I$. It then 
follows from the Artin--Rees lemma that $\Lambda\cap I^r=0$ for $r\gg0$, i.e., 
that the map \[A\to A/I^r\oplus\Gamma(\Spec A\setminus V(I),\roi_{\Spec A})\] is 
injective. Taking units and using the split determinant map $K_1\to\bb G_m$, 
which is an isomorphism for any local ring, proves that assertion (3)$_1$ is 
true. Hence assertion (1)$_1$ is true by Lemma \ref{lemma_equivalences}.
Since $A$ is local, the map $K_2(A) \to K_2(A/I)$ is surjective
and hence (2)$_1$ follows again from Lemma \ref{lemma_equivalences}.

We claim also that (1)$_0$, (2)$_0$, and (3)$_0$ are true; these are easy 
consequences of $K_0(A)=K_0(A/I^r)=\bb Z$.

Assuming that the resolution square satisfies the pro Mayer--Vietoris property 
for $K$-theory, Lemma \ref{lemma_equivalences}(iv) implies that assertions (4)$_1$  and 
(4)$_0$ are also true.
\end{exm}

\begin{exm}[$n<0$]
Assume $n<0$. Then (4)$_n$ becomes vacuously true since $K_n(X)=0$. 
Also, (1)$_n$ and (3)$_n$ become the identical statement that 
$K_n(A)\to K_n(A/I)$ is injective, since negative $K$-groups of rings are 
nil-invariant. Similarly, (2)$_n$ becomes the assertion that 
$K_n(A,I)\to K_n(X,Y_\sub{red})\cong K_{n+1}(Y_\sub{red})$ is injective.
\end{exm}

Based on the known results when $A$ is one-dimensional, which we will review in the following subsection, it remains plausible to conjecture that (2)$_n$ and (4)$_n$ 
might be true in general. Our goal in this note is to prove
two cases of (2)$_n$ and to prove (1)$_n$ -- (4)$_n$ for cone 
singularities. 

\subsection{The one-dimensional case}\label{sec:dim-1}
We now collect together the known results in the case 
that $A$ is one-dimensional and its desingularization is obtained by 
normalization.

Let $(A,\frak m)$ be a one-dimensional, Noetherian local ring such that 
$\Spec A\setminus\{\frak m\}$ is regular, i.e., $A_\frak p$ is a field for 
each minimal prime ideal $\frak p\subset A$. Let $\tilde{A_\sub{red}}$
denote the normalization of $A_{\rm red}$. Assume that the 
normalization map $A_\sub{red}\to B:=\tilde{A_\sub{red}}$ is finite. Then $B$ 
is a regular, one-dimensional, semi-local ring, and
\[
\xymatrix{
\Spec B/\frak mB\ar[r]\ar[d] & \Spec B\ar[d]\\
\Spec A/\frak m\ar[r] & \Spec A}
\]
is a resolution square which satisfies the pro Mayer--Vietoris property.
Assuming that $B$ contains a field, we
make the following assertions about 
this resolution square:
\begin{itemize}\itemsep0pt
\item (1)$_2$ is not necessarily true;
\item (2)$_n$ is true for all $n\in\bb Z$;
\item (3)$_2$ is not necessarily true;
\item (4)$_n$ is true for all $n\in\bb Z$;
\end{itemize}
Firstly, (4)$_n$ is true since the validity of Gersten's conjecture in the 
geometric case implies that $K_n(B)\to K_n(\Spec B\setminus V(\frak mB))$ is 
injective.

Secondly, it now follows from Lemma \ref{lemma_equivalences}(ii) that (1)$_2$ 
and (3)$_2$ are equivalent; in the next example we will provide a particular 
choice of $A$ for which (3)$_2$ fails.

We now show that (2)$_n$ is true for all $n\in\bb Z$; in fact, 
this was proved in \cite[Thm.~2.7]{Morrow_Birelative_dim1}, inspired by the 
case $n=2$ \cite[Thm.~2.9]{Krishna2005}, under the extraneous assumption that 
$A$ was reduced. Indeed, it was shown 
\cite[Corol.~2.6]{Morrow_Birelative_dim1} that the canonical map 
$K_n(B,\frak M)\to \{K_n(B/\frak M^r,\frak M/\frak M^r)\}_r$ is surjective for 
all $n\ge0$, where $\frak M=\sqrt{\frak mB}$ is the Jacobson radical of 
$B$; it is also surjective if $n<0$, since the codomain vanishes by nil-invariance of negative $K$-theory for rings. Hence~\eqref{eqn:pro-ex-1} breaks into short exact sequences, proving 
(2)$_n$.

\begin{exm}\label{exm:failure}
Let $k$ be a field of char $\neq2$ and $A=k[X,Y]_{(X,Y)}/(Y^2-X^2(X+1))$ 
the local ring at the singular point of the nodal curve $Y^2=X^2(X+1)$. Then 
$A$ is a domain, with finite normalization map since it is essentially of finite type over a field. 
However, assertion (3)$_2$ fails: the map 
$K_2(A)\to K_2(A/\fm^r) \oplus K_2({\rm Frac} A)$ is not injective for any 
$r\ge1$. The proof may be found in \cite[Prop.~2.12]{Morrow_Singular_Gersten} 
and
relies on the fact that 
$K_3(\tilde A)\to K_3(\tilde A/\fm\tilde A)=K_3(k)\oplus K_3(k)$ can be 
shown not to be surjective.
\end{exm}

\section{Vanishing of some relative Hochschild and
cyclic homology}\label{section_preliminary}
This section, where we
establish some vanishing results for relative Hochschild and cyclic
homology groups, is at the heart of the proofs of our main results.
The following preliminary lemma may be ignored by readers who are only 
interested in the case of finite type algebras 
(``smooth'' means ``localisation of smooth, finite type''):

\begin{lem}\label{lemma_relative_NP}
Let 
$f:A\to B$ be a surjection of 
regular, local $\bb Q$-algebras. Then it is 
possible to write $A=\indlim_iA_i$ and $B=\indlim_i B_i$ as filtered inductive 
limits of smooth, local $\bb Q$-algebras, in such a way that $f=\indlim_if_i$ 
where $f_i:A_i\to B_i$ are compatible surjections.
\end{lem}
\begin{proof}
Choose a regular system of parameters $t_1,\dots,t_d$ of $A$ such that 
$I:=\ker f$ is generated by $t_1,\dots,t_c$ for some $0\le c\le d=\dim A$.
According to N\'eron--Popescu desingularization, $A$ may be written as a 
filtered inductive limit $\indlim_{i\in I}A_i$ of smooth, local $\bb Q$-algebras 
such that the homomorphisms $\phi_i:A_i\to A$ are local. Let 
$\fm_i=\phi^{-1}_i(\fm_A)$ denote the maximal ideal of $A_i$. Possibly 
after discarding part of the bottom of this inductive system, we may assume 
that $I$ has a minimal element $i_0$ and that the elements $t_1,\dots,t_c$ may 
be lifted to $\tilde t_1,\dots,\tilde t_c\in \fm_{i_0}$.
Set $B_i = {A_i}/{(\tilde t_1,\dots,\tilde t_c)}$.

The choice of $t_1,\dots,t_c \in A$ implies that 
the images of $\tilde t_1,\dots,\tilde t_c$ in 
${\fm_i}/{\fm_i^2}$ are linearly independent over $A_i/\fm_i$.
In particular, their images in $A_i$ form part of a regular system of 
parameters. This proves the lemma.
\end{proof}

\begin{remk} It was remarked by the referee that \lemref{lemma_relative_NP}
holds for surjective maps $A \to B$ of regular algebras containing any field.
But we do not use this generalization here.
\end{remk}

\begin{prop}\label{proposition_smooth_vanishing}
Let $Y\into X$ be a closed embedding of regular, Noetherian $\Q$-schemes of 
finite Krull dimension. Then the pro abelian group 
$\{HC_n^i(rY,Y)\}_r$ vanishes for $0\le i<n$.
\end{prop}
\begin{proof}
By the Zariski descent of cyclic homology 
(see \cite[Thm~2.9]{CHSW} for schemes essentially of finite type over
a field and \cite{Weibel-1} for general Noetherian $\Q$-schemes),
we may assume that $X=\Spec R$ is affine, with $Y$ defined by an ideal 
$I\subseteq R$; note that $R$ and $R/I$ are regular.

The usual map of mixed complexes 
$(C_\blob^{\Q}(-),b,B)\to (\Omega_{-}^\blob,0,d)$ is an isomorphism on the 
associated Hochschild homologies both for $R$ and $R/I$, 
by the usual HKR theorem 
\cite[Thm.~3.4.4]{Loday1992}, and for the pro ring $R/I^\infty$ by the 
pro HKR theorem \cite[Thm.~3.23]{Morrow_pro_H_unitality}.
There are therefore induced isomorphisms of the associated cyclic homologies 
and of the relative homologies; in particular, 
\[
\{HC_n^{(i)}(R/I^r,I/I^r)\}_r\isoto 
\{H^{2n-i}(\ker(\Omega_{R/I^r}^\blob\to\Omega_{R/I}^\blob))\}_r
\] for $0\le i<n$.

Hence, to complete the proof, we may show that the canonical projection 
$\Omega_{R/I^r}^\blob\to\Omega_{R/I}^\blob$ is a quasi-isomorphism for each 
$r\ge1$. If $I$ is generated by a regular sequence and $R$ is essentially of 
finite type over a subfield $K\subseteq R$, whence $S:=R/I$ is formally smooth 
over $K$, then it is well-known (e.g., \cite[Lem.~II.1.2]{Hartshorne1975}) 
that the $I$-adic completion of $R$ is isomorphic to $S[[T_1,\dots,T_c]]$. In 
particular, $R/I^r\cong S[T_1,\dots,T_c]/(T_1,\dots,T_c)^r$ admits the 
structure of a positively graded $K$-algebra with degree zero component $S$, 
and so it follows from the Poincar\'e Lemma \cite[Corol.~9.9.3]{Weibel1994} 
that $\Omega_{R/I^r}^\blob\quis\Omega_{S}^\blob$ (even though the base field for these K\"ahler differentials is $\bb Q$, not $K$). In general, we can easily 
reduce to this case, by using Zariski descent to assume that $R$ is local and by then applying \lemref{lemma_relative_NP}.
\end{proof}

The following two results are a modification of the previous proposition  
when $Y$ is replaced by a normal crossing divisor. 
A {\em strict normal crossing divisor} on a regular affine scheme $\Spec(R)$ 
is a closed subscheme defined by a non-zero-divisor of the form 
$t_1\cdots t_c$, where $t_1,\dots,t_c\in R$ have the property that for each 
prime ideal $\fp\in V(t_1\cdots t_c)$, those of the $t_i$ which belong to 
$\fp R_{\fp}$ form part of a regular system of parameters of 
$R_{\fp}$. A {\em normal crossing divisor} on a regular scheme is a divisor 
which \'etale locally is a strict normal crossing divisor.

\begin{lem}\label{lem:sncd}
Let $R$ be a regular, local $\Q$-algebra and let $f\in R$ be a 
non-zero-divisor for 
which the de Rham differential $d:fR/f^2R\to\Omega_{R}^1\otimes_RR/fR$ is 
injective. Then:
\begin{enumerate}
\item $HH_n^{(i)}(R/fR)=0$ if $i\le n/2$ (unless $i=n=0$).
\item $\{HH_n^{(i)}(R/f^rR,fR/f^rR)\}_r=0$ if $i\le (n+1)/2$ 
(unless $i=n=1$ or ${i=n=0}$).
\item $\{HC_n^{(1)}(R/f^rR,fR/f^rR)\}_r=0$ if $n\ge 2$.
\end{enumerate}
\end{lem}
\begin{proof}
(i): Since $R$ is geometrically regular over $\Q$ (since it is regular) and $f$ is a 
non-zero-divisor, the cotangent complexes $\bb L_{R}$ and $\bb L_{R/fR|R}$ of 
$\Q\to R$ and $R\to R/fR$ are respectively equal to 
$\Omega_{R}^1$ and $(fR/f^2R)[1]$, whence it follows from the 
Jacobi--Zariski sequence that the cotangent complex $\bb L_{R/fR}$ of 
$\Q\to R/fR$ is quasi-isomorphic to the following chain complex of flat 
$R/fR$-modules: 
\[
0\leftarrow\Omega_{R}^1\otimes_RR/fR\leftarrow fR/f^2R\leftarrow0.
\]

More generally, for any $i\ge 1$, the exterior powers $\bb L_{R}^i$ and 
$\bb L_{R/fR|R}^i$ are respectively equal to $\Omega_{R}^i$ and 
$(f^iR/f^{i+1}R)[i]$, and it follows from the 
Kassel--Sletsj\o e spectral sequence \cite[Thm.~3.2]{Kassel1992} that 
$\bb L^i_{R/fR}$ is quasi-isomorphic to the following chain complex of flat 
$R/fR$-modules:
\begin{equation}\label{eqn:sncd-0}
0\leftarrow\Omega_{R}^i\otimes_RR/fR\leftarrow\Omega_{R}^{i-1}\otimes_RfR/f^2R
\leftarrow\cdots\leftarrow\Omega_{R}^1\otimes_Rf^{i-1}R/{f^iR}\leftarrow 
f^iR/f^{i+1}R\leftarrow 0
\end{equation}
(where $\Omega_{R}^{i-j}\otimes_Rf^jR/f^{j+1}R$ sits in degree $j$).

This presentation of the cotangent complex can also be derived from
\cite[Corol.~3.4]{CGG},
which holds more generally for complete intersection ideals in $R$ (see also
\cite[\S~5]{FT}).

Using~\eqref{eqn:sncd-0}, part of the Andr\'e--Quillen homology of 
$\bb Q\to R/fR$ can be written as
\[
D_n^i(R/fR) = \left\{
\begin{array}{ll}
\ker(d:f^iR/f^{i+1}R\to\Omega_{R}^1\otimes_Rf^{i-1}R/f^iR) &
\mbox{if $i =n$,} \\
0 & \mbox{if $i > n$}.
\end{array}
\right.
\]
It follows easily from the hypothesis on $f$ that this kernel is zero. 
Finally, recall that $D_n^i(R/fR)=HH_{n+i}^{(i)}(R/fR)$ to complete the proof.

(ii): There are short exact sequences of pro $R$-modules 
\[
HH_{n+1}^{(i)}(R/fR)\To \{HH_n^{(i)}(R/f^rR,fR/f^rR)\}_r\To 
\{HH_n^{(i)}(R/f^rR)\}_r,
\] 
where the right term vanishes for $i<n$ by the pro HKR theorem 
\cite[Thm.~3.23]{Morrow_pro_H_unitality}. Moreover, the left term vanishes if 
$i\le(n+1)/2$ by (i), and hence the central term also vanishes.

(iii): To save space, we will temporarily use the notation $HH_n^{(i)}$ for the pro abelian group 
$\{HH_n^{(i)}(R/f^rR,fR/f^rR)\}_r$ and similarly for cyclic homology. For any 
$n\ge 3$, we see from (ii) and the SBI sequence 
$HH_n^{(1)}\to HC_n^{(1)}\to HC_{n-2}^{(0)}= 0$ that $HC_n^{(1)}=0$. It remains to 
treat the difficult case when $n=2$. 

In the SBI sequence $HH_2^{(1)}\to HC_2^{(1)}\to HC_0^{(0)}\to HH_1^{(1)}$, it follows from (i) that the left term 
vanishes and the right term embeds into $\{HH_1^{(1)}(R/f^rR)\}_r=\{\Omega_{R/f^rR}^1\}_r$, where we have again applied the pro HKR theorem. 
Since $HC_0^{(0)}=\{fR/f^rR\}_r$ with the final arrow $HC_0^{(0)}\to HH_1^{(1)}$ 
corresponding to the de Rham differential, there is an induced isomorphism 
\[HC_2^{(1)}\isoto\{\ker(d:fR/f^rR\to\Omega_{R}^1\otimes_RR/f^{r-1}R)\}_r\] 
(here we have implicitly used the Leibniz rule to identify 
$\{\Omega_{R}^1\otimes_RR/f^{r-1}R\}_r$ and $\{\Omega_{R/f^rR}^1\}_r$). It 
follows easily from the hypothesis on $f$ that this kernel is zero, and so we 
deduce that $HC_2^{(1)}=0$.
\end{proof}

\begin{exm}\label{exm:sncd-1}
Suppose that $R$ is a regular $\Q$-algebra and that $t\in R$ is a 
non-zero-divisor for which $R/tR$ is also regular; then we claim that the 
hypothesis of the previous lemma is satisfied, i.e., that 
$d:tR/t^2R\to\Omega_{R}^1\otimes_RR/tR$ is injective.

By Lemma \ref{lemma_relative_NP} and the same reductions as in Proposition 
\ref{proposition_smooth_vanishing}, it is enough to prove that 
$d:TS[[T]]/T^2S[[T]]\to\Omega_{S[[T]]}^1\otimes_{S[[T]]}S$ is injective for any 
regular $\Q$-algebra $S$. But this is clear, since the canonical differential 
$d/dt$ induces a map of $S[[T]]$-modules $\Omega_{S[[T]]}^1\to S[[T]]$ such 
that the resulting composition 
\[TS[[T]]/T^2S[[T]]\to\Omega_{S[[T]]}^1\otimes_{S[[T]]}~S\to S[[T]]\otimes_{S[[T]]}S=S\] is given by $sT\mapsto s$.
\end{exm}

\begin{prop}\label{proposition_cyclic}
Let $X$ be a regular, Noetherian $\Q$-scheme of finite Krull dimension, and let 
$Y\into X$ be a normal crossing divisor. Then $\{HC_n^{(1)}(rY,Y)\}_r=0$ for 
all $n\ge2$.
\end{prop}
\begin{proof}
Since relative cyclic homology (and each of its Adams summands) satisfies \'etale descent 
(see \cite[Theorem~2.9]{CHSW}), the usual induction on the size of an affine 
cover (note that $X$ is quasi-separated) allows us to assume that 
$X=\Spec(R)$ is affine and that $Y$ is a strict normal crossings divisor, 
defined by $f=t_1\cdots t_c\in R$.

Then part (iii) of the previous lemma will complete the proof, as soon as we 
show that the de Rham differential $d:fR/f^2R\to\Omega_{R}^1\otimes_RR/fR$ is 
injective; to check this we may assume that $R$ is local. Considering the 
commutative diagram 
\[\xymatrix{
fR/f^2R \ar[r]^d\ar[d] &\Omega_{R}^1\otimes_RR/fR\ar[d]\\
\bigoplus_{i=1}^ct_iR/t_i^2R \ar[r]^<<<<<{d} &
\bigoplus_{i=1}^c\Omega_{R}^1\otimes_RR/t_iR,
}\]
it is enough to show that the left vertical and bottom horizontal arrows are 
injective. The left vertical arrow is injective since $R$ is a unique 
factorisation domain and the elements $t_1,\dots,t_c$ are all distinct
irreducibles. The bottom horizontal arrow is injective by the previous example.
\end{proof}

\section{The main results}\label{sec:MR}
We now collect together the vanishing results of Section \ref{section_preliminary} to prove
our main theorems.

\subsection{Proof of Theorem \ref{thm:Main-1}}
 Throughout this section we consider a quasi-excellent, 
Noetherian, local $\bb Q$-algebra $A$, an ideal $I\subseteq A$
and a resolution square as in~\eqref{eqn:ABlow-up}.
 
We begin with the following result about the relative $K$-groups (which in fact holds for any ideal of any local Noetherian ring):

\begin{lem}\label{lemma_K_theory}
For any $i> n\ge1$, the group $K_n^{(i)}(A,I)$ is torsion of bounded exponent.
\end{lem}
\begin{proof}
Let $\sT$ be the category of torsion groups of bounded exponent. 
By standard properties of lambda and Adams operators, the sequence of 
relative $K$-groups associated to $A\to A/I$ yields to an exact sequence 
\[
K_{n+1}^{(i)}(A) \To K_{n+1}^{(i)} (A/I) \To K_n^{(i)}(A,I)\To K_n^{(i)}(A)
\] 
in the category of abelian groups modulo $\sT$.

Moreover, since $A$ is local the group $K_n^{(i)}(A)$ is torsion of bounded 
exponent if $i>n$ by Soul\'e \cite[\S2.8]{Soule1985}; the same applies to 
$K_{n+1}^{(i)}(A/I)$ if $i>n+1$. 
This completes the proof in the case $i>n+1$, and shows that 
$K_n^{(n+1)}(A,I)= {\rm Coker}(K_{n+1}^{(n+1)}(A)\to K_{n+1}^{(n+1)}(A/I))$ modulo 
$\sT$. But this cokernel is zero since $K_{n+1}^{(n+1)}(A)=K_{n+1}^M(A)$ and 
$K_{n+1}^{(n+1)}(A/I)=K_{n+1}^M(A/I)$ modulo $\sT$ by 
Nesterenko--Suslin \cite{Suslin1989}.
\end{proof}

\begin{lem}\label{lem:Main-step-1}
Fix $n\ge 0$, and suppose that 
$\{K_{n+1}^{(i)}(rY,Y_{\red})\}_r=0$ for all $i\le n$. Then
\begin{enumerate}
\item the canonical map of pro abelian groups 
$K_{n+1}(X,Y_{\red})\to \{K_{n+1}(rY,Y_{\red})\}_r$ is surjective;
\item assertion (2)$_n$ is true, i.e., 
$K_n(A,I)\to K_n(A/I^r,I/I^r)\oplus K_n(X,Y_{\red})$ is injective for 
$r\gg0$.
\end{enumerate}
\end{lem}

\begin{proof}
(i): From the long exact sequence 
\[
\cdots\To K_{n+1}(X,Y_\sub{red})\To 
\{K_{n+1}(rY,Y_\sub{red})\}_r\xto{\{\bor_r\}_r}\{K_n(X,rY)\}_r\To\cdots,
\] we see that it is necessary and sufficient to prove that the boundary map 
$\{\bor_r\}_r$ is zero.

The square~\eqref{eqn:ABlow-up} satisfies the pro Mayer--Vietoris property by the results recalled in Section \ref{section_assertions}, and so the canonical map $\{K_n(A,I^r)\}_r\to\{K_n(X,rY)\}_r$ is an isomorphism. Passing to Adams eigenspaces yields isomorphisms $\{K_n^{(i)}(A,I^r)\}_r\to\{K_n^{(i)}(X,rY)\}_r$. The surjectivity of this latter map means that if we fix any $s\ge 1$, then there exists $s'\ge s$ such that \[\Im(K_n^{(i)}(X,s'Y)\to K_n^{(i)}(X,sY))\subseteq \Im(K_n^{(i)}(A,I^s)\to K_n^{(i)}(X,sY)).\] 
It follows from Lemma \ref{lemma_K_theory} that the right, and hence the left, 
side is a torsion group of bounded exponent if $i>n$.

Since $rY$ is a nilpotent thickening of $Y_\sub{red}$, the relative $K$-group $K_{n+1}(rY,Y_\sub{red})$ is a $\bb Q$-vector space which decomposes as a direct sum $\bigoplus_{i=0}^{n+1+\dim Y}K_{n+1}^{(i)}(rY,Y_\sub{red})$.
Our hypothesis is that this decomposition, as a pro abelian group over $r\ge1$, has no component in degrees $i\le n$. Hence there exists $s''\ge s'$ such that the canonical map $K_{n+1}(s''Y,Y_\sub{red})\to K_{n+1}(s'Y,Y_\sub{red})$ has image in $\bigoplus_{i={n+1}}^{n+1+\dim Y}K_{n+1}^{(i)}(s'Y,Y_\sub{red})$.

Assembling the conclusions of the two previous paragraphs, and noting that the boundary maps $\bor_r$ respect the Adams operators, we see that the image of the composition \[K_{n+1}(s''Y,Y_\sub{red})\To K_{n+1}(s'Y,Y_\sub{red})\xto{\bor_{s'}} K_n(X,s'Y)\To K_n(X,sY)\] is both divisible (being the image of a $\bb Q$-vector space) and a torsion group of bounded exponent; hence the image is zero. This means that the map $\{\bor_r\}_r$ of pro abelian groups is zero.
Claim (ii) follows from (i) and the Mayer--Vietoris sequence~\eqref{eqn:pro-ex-2}.
\end{proof}

The following is the first main result of this note, proving assertion (2)$_n$ in certain cases:

\begin{thm}\label{theorem_main}
Assume either that $n = 2$, or that $Y_{\red}$ is regular and $n\ge 2$. 
Then the canonical sequence of pro abelian groups 
\[
0\To K_n(A,I)\to \{K_n(A/I^r,I/I^r)\}_r\oplus K_n(X,Y_{\red})\To 
\{K_n(r Y,Y_{\red})\}_r\To 0
\]
is short exact.
\end{thm}
\begin{proof}
If $Y_\sub{red}$ is regular then it follows from Proposition \ref{proposition_smooth_vanishing} and the Goodwillie isomorphism $K_{n+1}^{(i+1)}(rY,Y_\sub{red}) \cong HC_n^{(i)}(rY,Y_\sub{red})$ that $\{K_{n+1}^{(i+1)}(rY,Y_\sub{red})\}_r$ 
vanishes for $i<n$. From \lemref{lem:Main-step-1}, it then follows that~\eqref{eqn:pro-ex-2}
breaks into short exact sequences in positive degrees, as required.

Dropping the regularity hypothesis on $Y_\sub{red}$, it follows from 
resolution of 
singularities (specifically, since we are not imposing any finite type 
hypotheses, from \cite[Thm.~2.3.6]{Temkin}) that there is an abstract blow-up 
square
\[\xymatrix{
Y'\ar[r]\ar[d] & X'\ar[d]\\
Y\ar[r] & X
}\]
in which $X'$ is regular and $Y'_\sub{red}$ is a normal crossing divisor on $X$. Repeating the argument of the previous paragraph with $X,Y$ replaced by $X', Y'$, Proposition~\ref{proposition_smooth_vanishing} replaced by Proposition~\ref{proposition_cyclic}, and $n$ replaced by $2$, it follows that $K_2(A,I)\to K_2(A/I^r,I/I^r)\oplus K_2(X',Y'_\sub{red})$ is injective for $r\gg1$. Hence the map $K_2(A,I)\to K_2(A/I^r,I/I^r)\oplus K_2(X,Y_\sub{red})$ is certainly injective. 
The desired short exact sequence now follows from~\eqref{eqn:pro-ex-2}
and assertion (2)$_1$, which was proved in Example~\ref{example01}.
\end{proof}

\subsection{Proof of Theorem \ref{thm:Cone-Intro}}
Let $k$ be a field of characteristic zero and $Y\into \P_k^N$ a smooth 
projective variety. Let $C\into\bb A_k^{N+1}$ be the cone over $Y$, and let 
$(A,\frak m)$ be the local ring at the unique singular point of $C$. In the course of proving our second main result, we will see that assertions (1)$_n$--(4)$_n$ are all true for the usual resolution of $A$:

\begin{thm}\label{thm:Cone-final}
For any $n\ge0$, the map 
\[
K_n(A)\To K_n(A/\fm^r)\oplus K_n(\Spec(A) \setminus\{\fm\})\] is injective for 
$r\gg 0$.
\end{thm}
\begin{proof}
Let $\check C\to C$ be the usual resolution of singularities, which is a line 
bundle over $Y$ in such a way that the zero section $\sigma:Y\into\check C$ 
is exactly the exceptional fibre of $\check C\to C$ over the singular point of
$C$.

Let $X=\check C\times_C\Spec(A)$ be the associated resolution of singularities 
of $\Spec(A)$. Diagrammatically we have the following Cartesian squares:
\[\xymatrix{
X\times_{\Spec(A)} \Spec(A/\fm)\ar[r]\ar[d] & X\ar[r]\ar[d] & \check C\ar[d]\\
\Spec(A/\fm) \ar[r] & \Spec(A)\ar[r] & C 
}\]
in which $(X\times_{\Spec(A)} \Spec(A/\fm))_\sub{red}=Y$ and the left square and 
outer rectangle are abstract blow-up squares. The left square, which is the usual resolution square for $A$, satisfies 
(2)$_n$ by Theorem \ref{theorem_main}.

Since the canonical inclusion $Y\inj X$ is split by 
$X\to \check C\xto{\pi}Y$, where $\pi$ denotes the line bundle structure map, 
we see that $K_{n+1}(X)\to K_{n+1}(Y)$ is surjective. So it follows from 
Lemma \ref{lemma_equivalences}(iii) that the left square satisfies (1)$_n$. Since we wish to prove (3)$_n$, it is now enough by Lemma \ref{lemma_equivalences}(ii) to prove that the left 
square satisfies (4)$_n$, which we will do by showing that 
$K_n(X)\to K_n(Y)\oplus K_n(X\setminus Y)$ is injective.

Suppose that $\al\in K_n(X)$ dies in $K_n(X\setminus Y)$. By comparing the 
localisation sequences
\[\xymatrix{
K_n(Y)\ar[r]^{\sigma_*}\ar@{=}[d] & K_n(\check C)\ar[r]\ar[d] & 
K_n(\check C\setminus Y)\ar[d]\\
K_n(Y)\ar[r]_{\sigma_*} & K_n(X)\ar[r] & K_n(X\setminus Y),
}\]
we see that $\al$ is in the image of $K_n(\check C)\to K_n(X)$. But  the 
composition $K_n(\check C)\to K_n(X)\to K_n(Y)$ is an isomorphism since 
$\check C$ is a line bundle over $Y$; so if we now also assume that $\al$ dies 
in $K_n(Y)$, then it follows that $\al=0$. This completes the proof.
\end{proof}

\end{document}